\documentclass{article}

\usepackage{graphicx} 
\usepackage{graphicx} 
\usepackage{a4wide}
\usepackage[T1]{fontenc}
\usepackage{a4wide}
\usepackage{lipsum}
\usepackage{mathtools}
\usepackage{amsthm}
\usepackage{amsfonts}
\usepackage{yfonts}
\usepackage{amssymb}
\usepackage{xcolor}
\usepackage{mathrsfs}
\usepackage{amsmath}
\usepackage{cite}
\usepackage{epsfig}
\usepackage{stmaryrd}
\usepackage[utf8]{inputenc} 
\usepackage[shortlabels]{enumitem}
\usepackage[T1]{fontenc}
\usepackage{caption}
\usepackage{soul}

\newcommand\NN{{\mathbb N}}
\newcommand\ZZ{{\mathbb Z}}

\usepackage[utf8]{inputenc}
\usepackage[utf8]{inputenc}
\usepackage{fullpage}
\usepackage{amsthm}
\usepackage[pdftex,colorlinks,citecolor=blue]{hyperref}
\usepackage{hyperref} 
\usepackage{cleveref}
\usepackage{amsmath, amssymb, amsbsy}
\usepackage{algorithm}
\usepackage{algcompatible}
\usepackage{algpseudocode}
\usepackage{tikz,caption}
\usetikzlibrary{shapes,arrows,calc}

\allowdisplaybreaks

\crefname{figure}{figure}{}
\Crefname{figure}{Figure}{}

\usetikzlibrary{patterns,calc}

\newcommand{\beq}[1]{\begin{equation}\label{#1}}
\newcommand{\enq}[0]{\end{equation}}

\newcommand{\cA}{\mathcal{A} }
\newcommand{\cB}{\mathcal{B} }
\newcommand{\cC}{\mathcal{C} }

\newcommand{\cP}{\mathcal{P} }

\newcommand{\cR}{\mathcal{R} }

\newtheorem{theorem}{Theorem}[section]

\newtheorem{lemma}[theorem]{Lemma}
\AddToHook{env/lemma/begin}{\crefalias{theorem}{lemma}}
\newtheorem*{claim*}{Claim}
\newtheorem{claim}[theorem]{Claim}
\AddToHook{env/claim/begin}{\crefalias{theorem}{claim}}

\AddToHook{env/proposition/begin}{\crefalias{theorem}{proposition}}

\AddToHook{env/corollary/begin}{\crefalias{theorem}{corollary}}

\AddToHook{env/fact/begin}{\crefalias{theorem}{fact}}
\newtheorem*{fact*}{Fact}

\theoremstyle{definition}

\newtheorem{defn}[theorem]{Definition}
\AddToHook{env/defn/begin}{\crefalias{theorem}{definition}}
\newtheorem{remark}[theorem]{Remark}
\AddToHook{env/remark/begin}{\crefalias{theorem}{remark}}
\newtheorem*{remark*}{Remark}

\newenvironment{subproof}[1][\proofname]{%
  \begin{proof}[#1]%
}{%
  \end{proof}%
}

\definecolor{andrey}{rgb}{0.3, 0.65,,0.2}
\definecolor{trees}{rgb}{1, .7,.5}
\definecolor{vertex}{rgb}{.65, 0,.85}

\usepackage[margin=2.5cm]{geometry}

\title{A linear upper bound on zero-sum Ramsey numbers of $d$-degenerate graphs in $\mathbb{Z}_p$}
\author{Andrey Shapiro\thanks{King's College London. E-mail: \tt{andrey.shapiro@kcl.ac.uk}}}
\date{}

\begin{document}

\maketitle
\begin{abstract}
 Let $p$ be a prime number and let $G$ be a graph on $n$ vertices and $m$ edges. The zero-sum Ramsey number of $G$ over $\mathbb{Z}_p$, denoted by $R(G, \mathbb{Z}_p)$, is the minimum $\ell\in \mathbb{N}$ such that for any edge-coloring $c:E(K_\ell)\to\mathbb{Z}_p$, there is a subgraph $G'\subset K_\ell$ isomorphic to $G$ and satisfying $\sum_{e\in E(G')}c(e)=0$.
 
 We prove that if $G$ is a $d$-degenerate graph, then $R(G, \mathbb{Z}_p)\leq n + (3+3d)p$ so long as $m\geq 2pd(d+1)^2$, $p$ divides $m$, and $2d<p$. This generalizes a result by Colucci and D'Emidio on $1$-degenerate graphs. 
    \end{abstract}
\section{Introduction}

The question of how large a disordered system must be before it exhibits a desired ordered subsystem lies at the core of Ramsey theory. This main idea lends itself to many formulations. In the context of graph theory one such formulation is: given a graph $G$ and a color set $C$, determine the least $\ell\in \mathbb{N}$ such that any edge-coloring $c:E(K_\ell)\to C$ of the complete graph will accommodate some injection $f:V(G)\to V(K_\ell)$ such that $f$ exhibits some desired property under the coloring $c$.

In the classical form $C=\{\text{blue},\text{red}\}$ and the goal is for $f$ to be monochromatic, that is, $\{f(x),f(y)\}$ has the same color for all edges $\{x,y\}\in E(G)$. Ramsey's theorem \cite{ramsey} states that the accompanying Ramsey number (the least $\ell \in \mathbb{N}$ such that such an injection into $K_\ell$ always exists) is finite. We denote the Ramsey number by $R(G)$, and when $C$ is expanded to contain $k$ colors the $k$-color (classical) Ramsey number is denoted by $R_k(G)$. Since the classical formulation, the problem has taken on many variations that change the requirement on $f$ (see, for example, Section 3 of the survey by Conlon, Fox, and Sudakov \cite{CFS15}). In this work we are looking at a variation with an algebraic flavor, known as \emph{zero-sum} Ramsey numbers. 

\begin{defn}\label{def:zero_sum}
    Given a graph $G$ and a finite abelian group $\Gamma$, the \emph{zero-sum Ramsey number} of $G$ over $\Gamma$, denoted by $R(G, \Gamma)$,  is the minimum $\ell\in \NN$ such that for any edge-coloring $c:E(K_\ell)\to \Gamma$, there exists a \emph{zero-sum} injection $f:V(G)\to V(K_\ell)$. That is, $f$ satisfies $$\sum_{\{x,y\}\in E(G)} c(\{f(x),f(y)\})=0_\Gamma.$$ 
\end{defn}

Ramsey's theorem has an earlier counterpart outside the context of graphs in the form of Van der Waerden's theorem \cite{zbMATH02580955} which states that for each $r$ and $k$ there is some $n=W(r,k)$ such that any $r$-coloring of $[n]\coloneqq \{1,2,...,n\}$ contains in it a monochromatic arithmetic progression of length $k$. Similarly, the story of zero-sum Ramsey theory on graphs starts outside of graphs in the form of the Erd\H{o}s-Ginzburg-Ziv theorem \cite{EGZ}, which states that given $2m-1$ elements in $\mathbb{Z}_m$, there is a collection of $m$ of them summing to zero. These results inspired the above notion of $R(G,\Gamma)$ which was introduced by Bialostocki and Dierker \cite{BD92}. It is important to note that although $R_k(G)$ is finite for all $G$ and even any finite size of $C$, the same is not true for $R(G,\Gamma)$. Indeed, recalling that the \emph{exponent} $\exp(\Gamma)$ is the smallest integer $r>0$ such that $g\cdot r=0_\Gamma$ for all $g\in \Gamma$, we require that $\exp(\Gamma)$ divides $e(G)$, the number of edges in $G$. Without this, for any $\ell\in \mathbb{N}$ we could define the constant coloring $c(\{x,y\})=g$ where $g\in \Gamma$ has some order $r$ that does not divide $e(G)$. Then for any embedding, the sum of the edges would be $g\cdot e(G)\neq 0_{\Gamma}$. On the other hand, it is easy to see that with this requirement, $R(G, \Gamma)\leq R_{|\Gamma|}(G)$ and therefore $R(G, \Gamma)$ is finite. Furthermore, when $p=e(G)$, we get that $R(G,\mathbb{Z}_p)\geq R(G)$ by treating a red-blue coloring as a $0$-$1$ coloring. Hence, the zero-sum Ramsey number is fundamentally linked to the classical Ramsey number.

A breakthrough result by Lee \cite{Lee17} (resolving a long-standing conjecture of Burr and Erd\"os) states that $R(G)$ is linear in $n$ for a particular class of graphs, called $d$-degenerate graphs: 
\begin{defn}
    A graph $G$ is $d$-degenerate if its vertex set can be ordered $V(G)=\{v_1,...,v_n\}$ so that for all $i\in [n]$, $|\{v_iv_j\in E(G):j>i\}|\leq d.$
\end{defn}
Preceding the result in \cite{Lee17} was the work by Chv\'atal,
R\"odl, Szemer\'edi, and Trotter \cite{ChRSzT} which showed that for all $G$ with maximum degree $\Delta$, $R(G)$ is linear in $n$ (here, and for the rest of the paper, $n\coloneqq |V(G)|$ and $m\coloneqq e(G)$). Likewise, the current author with Katz, Lian, and Malekshahian \cite{bounded_degree} showed that $R(G,\Gamma)$ is linear in $n$ in the regime where $G$ has maximum degree $\Delta$ and $|\Gamma|$ divides $m$. Hence, mirroring the results on classical Ramsey numbers, it is natural to conjecture (Conjecture 1 in \cite{bounded_degree}) that $R(G,\Gamma)$ is linear in $n$ for all $d$-degenerate graphs (when $|\Gamma|$ divides $m$). 

In the context of graphs without isolated vertices, the work by Colucci and D'Emidio \cite{CD25} took the first step in proving this for $d=1$ and cyclic groups $\Gamma=\mathbb{Z}_p$ where their result implies that for all forests (i.e. $1$-degenerate graphs) $R(G,\Gamma) \leq n + 9p-12$ under the condition that $m=\Omega(p^2)$. We improve this to all $d$-degenerate graphs, and significantly improve the constraint on $m$ to be $\Omega(p)$. We also offer a slight improvement to the upper bound on $R(G,\mathbb{Z}_p)$ when $d=1$. Our main result, then, is the following theorem:

\begin{theorem}\label{mainT}
Let $G$ be a $d$-degenerate graph on $n$ vertices and $m$ edges and let $p\leq m/(2d(d+1)^2)$ be a prime dividing $m$ such that $2d<p$. Then $$R(G,\mathbb{Z}_p)\leq n+(3+3d)p.$$
\end{theorem}

Notice that unlike the work on bounded degree graphs \cite{bounded_degree}, we do not allow $|\Gamma|=|\mathbb{Z}_p|=m$. This is because some sort of separation between $p$ and $m$ is necessary for the constant in front of $p$ to remain small (whereas in \cite{bounded_degree} they show that for bounded degree graphs, $R(G,\Gamma)\leq C\cdot n$ for some large constant $C$ super-exponential in $\Delta$). Indeed, even in the case of bipartite graphs with maximal degree $\Delta$, the results due to Graham et al. \cite{GRRR01} tell us that there exists a graph $G$ such that $ C^\Delta n\leq R(G)$. But then recall that when $p=m$, $R(G)\leq R(G,\mathbb{Z}_p)$.
Combining this with our result shows that if $p_1=m$ and $p_2\leq m/2d(d+1)^2$ are primes, then there exist $d$-degenerate graphs $G$ such that  $$R(G,\mathbb{Z}_{p_1})=2^{\Omega(d)}\cdot n \text{ \quad \quad while \quad \quad }R(G,\mathbb{Z}_{p_2})< 2\cdot n.$$

In general, the strength of the known bounds depends quite heavily on the relative size of $|\Gamma|$ and $m$. Notably, in the setting where $\Gamma=\ZZ_p$ is a cyclic group, exact results are known for all graphs when $p =2$ \cite{Caro94}, and for forests (i.e. $1$-degenerate graphs) when $p=3$ \cite{ADP25}. However, as is most often the case in Ramsey theory, most results in this area are not exact and aim to establish asymptotic bounds on $R(G,\Gamma)$. For a broader survey on these and related results see \cite{Carosurvey} and also \cite{CD25, Caro92, AC93, bounded_degree}.

\section{Theorem \ref{mainT}}
In this section we will introduce the key lemmas and then use them to prove our main result (\Cref{mainT}).
Then in \Cref{sec:lemma} we will prove the main embedding lemma (\Cref{lem:embedding}). We begin with a fundamental result from additive number theory that will be central to our approach:

\begin{theorem}[Cauchy-Davenport \cite{Cauchy}]\label{CD}

Let $p$ be a prime and let $A_1,...,A_s\subset \mathbb{Z}_p$ be non-empty. Then $$|A_1+...+A_s|\geq \min \{p,|A_1|+...+|A_s|-s+1\}.$$ 
    
\end{theorem}

\subsection{Terminology}
Before giving a brief overview of the strategy for the proof of \Cref{mainT}, we introduce some useful terminology.

Let $R$ be a set of vertices. Then given two vertices $x,y\in R$ we denote the pair $\{x,y\}$ by $xy$. Further, we can treat $R$ as the vertex set of a complete graph and write $E(R)\coloneqq\{uv:u\neq v\in R\}$. Given a subset of vertices $X=\{x_1,x_2,...,x_s\}\subset R$ and an edge coloring $c:E(R)\to \mathbb{Z}_p$, we define $c(yX)=c(Xy)\coloneqq \sum_{i=1}^s c(yx_i)$. Further, given an additional vertex $z\in R$ we define $c(Xyz)=c(Xy)+c(yz)$. 

Likewise, given a vertex coloring $\cC:R\to \mathbb{Z}_p$ we define $\cC(X)\coloneqq \sum_{i=1}^s \cC(x_i)$. Lastly, for $G$ a graph, $Y\subset V(G)$,  and $v\in V(G)\setminus Y$, we can define $d_G(v)$ to be the degree of $v$ in $G$ and $d(v,Y)\coloneqq |\{v'\in Y:vv'\in E(G)\}|$ to be the degree of $v$ into $Y$. Lastly, we denote $N(v)\coloneqq \{v'\in V(G): vv'\in E(G)\}$ to be the neighborhood of $v$ in $G$.

\subsection{Proof overview and key lemmas}

Let us now go over the main ideas behind the proof of \Cref{mainT}. For the rest of this section, let $G$ be a $d$-degenerate graph with $m$ edges and let $p\leq m/(2d(d+1)^2)$ be a prime dividing $m$ such that $2d<p$. 

Suppose we are given a vertex set $R$ and an edge coloring $c:E(R)\to \mathbb{Z}_p$. First, consider some \textbf{independent} set $J\subset V(G)$ and let $f:V(G) \to R$ be an injection with $\sum_{xy\in E(G)}c(f(x)f(y))=g$. Then suppose that for each $v\in J$ we could choose to modify $f$ into $f_v$ in such a way that $\sum_{xy\in E(G)}c(f_v(x)f_v(y))=g+\delta(v)$ for some $\delta(v)\neq 0$. Further suppose that any number of these modifications can be made and all are independent. That is, for each $W\subset J$, we have an injective $f_{W}$ such that $\sum_{xy\in E(G)}c(f_W(x)f_W(y))=g+\sum _{v\in W}\delta(v)$. Then as long as $|J|\geq p-1$, by a direct application of \Cref{CD} with $A_v=\{0,\delta(v)\}$, we get that there is some $W\subset J$ such that $\sum_{v\in W}\delta(v)=-g$ and so $\sum_{xy\in E(G)}c(f_W(x)f_W(y))=0$. Thus, $f_W$ is our desired zero-sum embedding.

Our proof will use \Cref{blueprint_extraction} to find an appropriate $J\subset V(G)$ while \Cref{lem:embedding} will use $J$ in order to find a suitable $f$ along with choices for its modification, or, if this is impossible, find a monochromatic embedding of $V(G)$ into $R$.  

\begin{lemma}\label{blueprint_extraction}
If $G$ is a $d$-degenerate graph with $m$ edges then there is an independent set $J\subset V(G)$ of size at least $m/d(d+1)^2$ where $1\leq d_G(v)\leq 2d$ for all $v\in J$.
\end{lemma}
\begin{proof}
Let $G'$ be the induced subgraph of $G$ on non-isolated vertices and say that $|V(G')|=n'$. We recall that we can order $V(G')=\{v_i,...,v_{n'}\}$ such that for all $i\leq n'$, $|\{j>i:v_jv_i\in E(G')\}|\leq d$. Let $x$ be the number of vertices in $V(G')$ of degree at most $2d$ and let $y=n'-x$. For each vertex $v_i$ with  degree $d(v_i)>2d$ we have $$|\{j<i:v_jv_i\in E(G')\}|\geq d+1.$$ This gives us that $y(d+1)\leq n'd$ and hence $x\geq n'/(d+1)$. Let $V_0$ be the set of vertices with degree at most $2d$ and observe that we can write $V_0=\{w_1,...,w_{x}\}$ such that for all $i\leq x$, $|\{j>i:w_jw_i\in E(G')\}|\leq d$. Then, if we in order extract vertices $w_i$ from $V_0$ and place them into $J$ and then remove from $V_0$ the neighborhood $N(w_i)$, we are guaranteed to at each step shrink $V_0$ by at most $d+1$ vertices and so at the end $|J|\geq  |V_0|/(d+1)\geq n'/(d+1)^2> m/d(d+1)^2$ since $m/n'< d$.
\end{proof}

Note that it is crucial that these vertices have degree at least 1 since different embeddings of a vertex of degree $0$ will never yield different sums. The reason why the vertices need to be of low degree will become evident in the proof of the following lemma:

\begin{lemma}\label{lem:embedding} Let $J=\{v_1,...,v_s\}\subset V(G)$ be an independent set of $s=2p$ vertices so that $1\leq d_G(v_i)\leq 2d$ for all $v_i\in J$. Further, let $R$ be a vertex set of size $n+(3+3d)p$ and let $c:E(R)\to \ZZ_p$ be an edge coloring. 

Either there is a monochromatic copy of $G$ in $R$ or there exist injections $f:V(G)\setminus J\to R$, and $h,h':J\to R$ satisfying the following
    \begin{enumerate}
        \item $f(V(G)\setminus J)\cap (h(V)\cup h'(V))=\emptyset$,\label{item1a}
        \item $h(v)\neq h'(v')$ for all $v\neq v'\in J$, and \label{item1b}
        \item $ |\{i\in [s]:c(h(v_i)f(N_i))\neq c(h'(v_i)f(N_i))\}| \geq p$.\label{item2}
    \end{enumerate}
\end{lemma}

The proof of \Cref{lem:embedding} will follow in \Cref{sec:lemma}. Note that the modification to $f$ we mentioned in the proof overview is realized by condition \ref{item2} and the choice of defining $f(v_i)=h(v_i)$ or $f(v_i)=h'(v_i)$ for any given $v_i\in J$. Further, in the proof overview we said that we need $|J|\geq p-1$. However, as it turns out, this would not give us enough flexibility and it is necessary to occasionally allow for some $v\in J$ not to provide a choice for a modification of $f$. Hence, we require $|J|\geq 2p$.

\subsection{Proof of Theorem \ref{mainT}}
We now prove \Cref{mainT} using \Cref{blueprint_extraction}, \Cref{lem:embedding}, and \Cref{CD}. 

\begin{proof}[Proof of \Cref{mainT}]

Let $R$ be a vertex set of size at least $n+(3+3d)p$ and let $c:E(R)\to \ZZ_p$ be an edge-coloring. We begin by applying 
\Cref{blueprint_extraction} to $G$ to get $J=\{v_1,...,v_s\}\subset V(G)$, an independent set of size $s\coloneqq 2p$ where each vertex has degree at most $2d$ but degree at least $1$. For each $v_i\in J$ denote $N_i\coloneqq N(v_i)$. We can then define $$U\coloneqq \bigcup_{i\in [s]}N_i  \quad \text{ and } \quad Z\coloneqq V(G)\setminus (J\cup U)$$ noting that $(U\cup Z)\cap J=\emptyset$. Then we apply \Cref{lem:embedding} to either get the embedding functions $f:U\cup Z\to R$ and $h,h':J\to R$, or a monochromatic copy of $G$ in $R$. In the latter case, the monochromatic copy of $G$ is a witness to a zero-sum embedding of $G$ because $p$ divides $m$. Then let us consider the former case.

Consider the set of sequences $\overline u=(u_i)_{i\in [s]}$ where each $u_i\in \{h(v_i),h'(v_i)\}$. For each such $\overline{u}$  let us denote $c(\overline u)\coloneqq \sum_{i\in [s]}c(u_if(N_i))$. Applying \Cref{CD} with $A_i\coloneqq \{c(h(v_i)f(N_i)),c(h'(v_i)f(N_i))\}$ we get $$p=\min\{p,s+p-s+1\}\leq \left|\sum_{i\in [s]} A_i\right|= |\{c(\overline u):\overline u\in \{h(v_1),h'(v_1)\}\times...\times \{h(v_s),h'(v_s)\}\}|$$ by condition \ref{item2} of \Cref{lem:embedding}.
Thus, for each $g\in \mathbb{Z}_p$ there is some sequence $\overline g \in \{h(v_1),h'(v_1)\}\times...\times \{h(v_s),h'(v_s)\}$ such that $c(\overline g)=g$. Next define $$z\coloneqq \sum_{\substack{x,y\in U\cup Z\\ xy\in E(G)}}-c(f(x)f(y))$$ and consider the corresponding sequence $\overline z = (z_i)_{i\in [s]}$ for which $c(\overline{z})=z$. We can then extend $f$ to $J$ with $f(v_i)\coloneqq z_i$. Note that $f$ remains injective because of conditions \ref{item1a} and \ref{item1b} of \Cref{lem:embedding} and the fact that $h$ and $h'$ are injections themselves. Further, $$\sum_{xy\in E(G)}c(f(x)f(y))=\sum_{i\in [s]}c(z_if(N_i))+\sum_{\substack{x,y\in U\cup Z\\ xy\in E(G)}}c(f(x)f(y))=z-z=0.$$
Thus, $f$ is a zero-sum embedding of $G$ into $R$.
\end{proof}

\section{Proof of Lemma \ref{lem:embedding}}\label{sec:lemma}

In the rest of this paper we prove \Cref{lem:embedding} thereby concluding our proof of \Cref{mainT}. The proof of \Cref{lem:embedding} can be broken down as follows. First we try to define $f$, $h$, and $h'$ into $R$. We will do this by embedding one element of $U$ at a time and embedding elements $v_i\in J$ (i.e. defining $h$ and $h'$ for $v_i$) only when the last element of $N_i$ is about to be embedded. If this process fails, then the set of unused vertices in $R$, which we will call $R(\tau)$, will exhibit some strong properties under the coloring $c$. Notably, we will have a set of $k\leq 2p$ vertex colorings such that each vertex will be strongly regular in at least $k/2$ of the colorings (for a precise definition of what we mean by strongly regular see \Cref{rem:C-C}). Keeping this property in mind, we will again try to define $f$, $h$, and $h'$ into $R(\tau)$. This time we will do this by embedding all of $U$ from the start into a set of vertices that are all strongly regular in some vertex coloring $\cC_\rho$. Then we will define $h$ and $h'$ one element of $J$ at a time. If this fails, we will be able to immediately argue that on the remaining vertices, $c$ is almost monochromatic (meaning that each vertex sees at most $k/2$ edges not in the majority color).
We conclude with an application of \Cref{lem:mono} which finds a monochromatic embedding of $G$ into $R(\tau)$ using $G$'s $d$-degeneracy property. 

\subsection{Embedding a $d$-degenerate graph}\label{sec:d-d}

Although we will only use the following lemma at the very end, we will present it now because it is self-contained and the second and last time we will use the $d$-degeneracy of $G$ (the first being \Cref{blueprint_extraction} where we found a sufficiently large independent set of non-isolated low-degree vertices). This lemma will allow us to later embed $G$ monochromatically in the case that the edge-coloring $c$ has some sufficiently high density in one color on some sufficiently large subset of vertices $\cR\subset R$.

\begin{lemma}\label{lem:mono}
    Let $G$ be a $d$-degenerate graph on $n$ vertices, and $H$ a graph such that for any collection of $d$ vertices $Y\subset V(H)$, their joint neighborhood $$N(Y)\coloneqq \bigcap_{y\in Y} N(y)$$ satisfies 
    $|N(Y)|\geq n-|Y|.$ Then there exists a subgraph $H'\subset H$ that is isomorphic to $G$.
\end{lemma}

\begin{remark}
    When $d=1$ this lemma is just the fact that any forest on $n$ vertices can be embedded into a graph with minimum degree $n-1$.
\end{remark}

\begin{proof}
$G$ is $d$-degenerate, so there is an ordering of vertices $v_1,...,v_n$ such that for all $j\in [n]$, $Y_j\coloneqq \{v_\ell \sim v_j:\ell>j\}$ has size at most $d$. So, starting with $i=n$, set $f:\emptyset \to \emptyset$ and working backwards through to $i=1$, let us embed each $v_i$ into some new $f(v_i)=w_i\in N(Y_i)\setminus f(\{v_{i+1},...,v_n\})\neq \emptyset$ (recalling that $|N(Y_i)|\geq n$). The result of this process is an embedding of $G$ into $H$ thereby concluding the proof. 
\end{proof}

\subsection{Embedding into $R$}\label{sec:intoR}

We will now prove \Cref{lem:embedding} by defining two processes that attempt to construct our desired $f$, $h$, and $h'$. If both of these processes fail, then we will be able to conclude via \Cref{lem:mono}. 

\begin{proof}[Proof of Lemma \ref{lem:embedding}]

Recall that $J=\{v_1,...,v_{s}\}\subset V(G)$ is an independent set where each $v_i\in J$ has $1\leq d_G(v_i)\leq 2d$. Also recall that $|R|\geq n+(3+3d)p$ and $c:E(R)\to \ZZ_p$ is an edge coloring. Let us also  define $$U\coloneqq \bigcup_{i\in [s]}N_i  \quad \text{ and } \quad Z\coloneqq V(G)\setminus (J\cup U)$$ as in the proof of \Cref{mainT}. 
We start with empty functions $f,h,h':\emptyset \to \emptyset$ and grow the domain of $f$ to $U\cup Z$ and the domains of $h$ and $h'$ to $J$. We order the elements of $U=\{u_1,...,u_t\}$ in descending order by their degree into $J$. That is, for each $0<j<j'\leq t$, $d(u_j,J)\geq d(u_{j'},J)$. Further, for each $i\in [s]$ we define $\ell_i$ to be such that $u_{\ell_i}$ is the last element in $N_i$ (under the ordering on $U$). Using the definition of $\ell_i$ for each $j\in [t]$ we define $$\omega(j)\coloneqq \{v_i\in J:\ell_i=j\}.$$

Our process will work by embedding each $u_j$ in order, and at each $j$ where $\omega(j)\neq \emptyset$ we will try to define $h$ and $h'$ on the vertices $v_i\in \omega(j)$. However, it is key that we will not require 
\begin{equation}\label{eq:ineq}
c(h(v_i)f(N_{i}))\neq c(h'(v_i)f(N_i))
\end{equation}
\textbf{for each} $v_i\in \omega(j)$. Instead, we only require that at least half of the $v_i\in \omega(j)$ satisfy \cref{eq:ineq}. Recall that we only need for $p$ of the $v_i$ across all $i\in [s]$ to satisfy \cref{eq:ineq}. Then we want to ensure that as soon as this condition is met we stop the process in order to use the vertices in $R$ efficiently. Likewise, for any $v_i$ that does not satisfy \cref{eq:ineq}, we will ensure that $h(v_i)=h'(v_i)$. To this end, we take $t'\leq t$ to be smallest integer for which  $$\psi(t')\coloneqq\sum_{j=1}^{t'}\left\lceil\frac{|\omega(j)|}{2}\right\rceil\geq p.$$ Note that $t'$ exists because $\psi(t)\geq 2p/2 =p$. Then for each $j\in[t'-1]$ let $k_j\coloneqq \lceil |\omega(j)|/2\rceil$ and set $k_{t'}\coloneqq p-\psi(t'-1)$. Observe that $\sum_{j\in[t']}k_j = p$. Finally, let us define the subset of $J$ which we will try to embed in our process: $$J'=\bigcup_{j\in [t']}\omega(j).$$ With these notions defined, we can describe the embedding process:

First, let us set $R(1)=R$. Next, we will visit the vertices $u_1,u_2,...,u_{t'}$ in order, and at each step $j$, do the following:
\begin{enumerate}
    \item if there is no $i\leq s'$ such that $u_j\in N_i$ then update $Z\gets Z\cup \{u_j\} $ and $R(j+1)\coloneqq R(j)$. 
    \item Otherwise, if $\omega(j)=\emptyset$ we take any vertex $u_j^*\in R(j)$, define $f(u_j)=u_j^*$, and set $R({j+1})\coloneqq R(j)\setminus \{u_j^*\}$.
    \item If $\omega(j)=\{a_1,...,a_{|\omega(j)|}\}\neq \emptyset$, temporarily set $k=|\omega(j)|$ and find some $u_j^*,w_1,w_1',...,w_k,w_k'\in R(j)$ such that:
    
    \begin{enumerate}
        \item $\{u_j^*,w_{i_1},w_{i_1}'\}\cap \{w_{i_2},w'_{i_2}\}=\emptyset$ for all $i_1\neq i_2\in [k]$,
        \item $w_i=w_i'$ for $k-k_j$ values of $i\in [k]$ 
        \item and $c(u^*_jw_i) + c(w_if(N_{a_i}\setminus \{u_j\})) \neq c(u^*_jw'_i)+c(w_i'f(N_{a_i}\setminus \{u_j\}))$ for $k_j$ values of $i\in [k]$.
    \end{enumerate}
    If this is not possible we set $\tau\coloneqq j$, and terminate the process.
    Otherwise, define $f(u_j)=u^*_j$, and for all $i\in [k]$ define $h(a_i)=w_i$ and $h'(a_i)=w'_i$. Lastly, set $R({j+1})\coloneqq R({j})\setminus \{u^*_j,w_1,w_1',...,w_k,w_k'\}.$
\end{enumerate}
Finally, after step $t'$, we take any $Z^*\subset R(t'+1)$ of size $|Z|+(|U|-t')$ and extend $f$ bijectively $$f:Z\cup (U\setminus \{u_1,...,u_{t'}\})\to  Z^*.$$ Also take any subset $J^*\subset R({t'+1})\setminus Z^*$ of size $s-s'$ and extend $h$ and $h'$ bijectively (and identically) $$h,h':J\setminus\{v_1,...,v_{s'}\}\to J^*.$$

If the process finishes successfully we will have injective mappings $f:U\cup Z\to R$, $h:J\to R$, and $h':J\to R$. Naturally, the process ensures that $f(U\cup Z)\cap(h(J)\cup h'(J))=\emptyset$ and $h(v)\neq h'(v')$ for all $v\neq v'\in J$. Furthermore, by construction, we have that $$p= \sum_{j\in[t']} k_j=|\{i\in [s']:c(h(v_i)f(N_{i}))\neq c(h'(v_{i})f(N_{i}))\}|$$ which gives us the desired result.

If, on the other hand, the process was terminated prematurely, it is only because for some $\tau \leq t'$, $\omega(\tau )=\{v_{a_1},...,v_{a_k}\}$ and we were not able to find appropriate $u_{\tau }^*$ and $w_1,w_1',...,w_k,w_k'$ in $R(\tau)$. In this case, for each $i\in [k]$ and any $w\in R(\tau)$, let us define the vertex coloring $$\cC_i(w)\coloneqq c(wf(N_{a_i}\setminus \{u_{\tau }\}))$$ so that for all $u, w_1,w_1',...,w_k,w_k'\in R(\tau)$ (where $u$ is distinct from each $w_i$ and $w_i'$) there is some $I\subset [k]$ of size $|I|>k-k_\tau$ such that $c(uw_i)+\cC_i(w_i)=c(uw_i')+\cC_i(w_i')$ for all $i\in I$. Then set $k'\coloneqq k_{\tau }-1$. 

\begin{claim}$$|R(\tau)|\geq n+ p +(2+3d)k'$$
\end{claim}

\begin{subproof}[Proof of claim]

Let us denote $r\coloneqq d(u_\tau,J)$. First, suppose that $r>1$. We recall that the vertices of $U$ are ordered decreasingly by their degree into $J$ and so $k=|\omega(\tau)|\leq r\leq d(u_j,J)$ for all $j\leq \tau$. Letting $Y$ be the set of $j\in [\tau-1]$ such that $u_j\in N_i$ for some $i\leq s'$, we get that $|Y|<\tau \leq  4dp/r$. Furthermore, we obtained $R(\tau)$ by removing from $R$ the $|Y|$ vertices we used for embedding $\{u_j:j\in Y\}$, by removing the image of $h$, and by removing the image of $h'$ (this removes at most $p-k'$ more vertices). Thus we get:

\begin{align}
    |R(\tau)|- (n+ p +(2+3d)k')&\geq |R|-|Y|-(2p-k)-(p-k')-(n+p+(2+3d)k')\\&\geq 1+(-1+3d)(p-k')-4dp/r.\label{R_tau_small}  
\end{align} Recalling that $k'\leq \lceil k/2\rceil-1\leq \lceil r/2\rceil-1$, we may assume that $k'= \lceil r/2\rceil-1$. Further, the function $\zeta(x) = 3\lceil x/2\rceil+4p/x$ has the property that for $x$ an integer greater than $2$, $\zeta(x)\leq  (\zeta(x+2)+\zeta(x-2))/2$. Thus, expression \ref{R_tau_small} is minimized when $r=2$, $r=3$, $r=2p$, or $r=2p-1$. When $r= 2$ we have $k'=0$ and so the expression becomes $1+(-1+3d)p-2dp>  0$ as desired. If $r=3$ we instead get $1+(-1+3d)(p-1)-4dp/3>  0$. On the other hand, when $r\in \{2p-1,2p\}$, $k'=p-1$ and so the expression becomes $1+(-1+3d)-4dp/r\geq 3d-4dp/(2p-1)\geq d-2d/(2p-1)> 0$ as desired.

Now instead suppose that $r=1$. Let $Z_1$ be the set of $i\leq s'$ such that $|\omega(\ell_i)|=1$, let $Z_2$ be the set of $i$ such that $|\omega(\ell_i)|\geq 2$ (since $d(u_\tau,J)=1$ we have that all such $i$ will be less than $s'$), and let $Z'_1\coloneqq [s]\setminus (Z_1\cup Z_2)$ be the remaining $i$. Then we have that
\begin{equation}\label{eq:simple_sum}
|Z'_1|+|Z_1|+|Z_2|=s=2p.
\end{equation}
Observe that by the definition of $\psi(t')$ and the fact that $ d(u_{t'},J)\leq d(u_\tau,J)=1$,
\begin{equation}\label{eq:psi_sum}
p=\psi(t')=\sum_{j=1}^{t'} \textbf{1}(|\omega(j)|=1) +(\lceil|\omega(j)|/2\rceil)\cdot\textbf{1}(|\omega(j)|>1)\geq |Z_1|+\sum_{j=1}^{t'} \textbf{1}(|\omega(j)|>1)\geq   
|Z_1|+|Z_2|/2
\end{equation}
since for each $j$ such that $|\omega(j)|>1$ we have at least two corresponding $i,i'\in Z_2$ such that $j=\ell_i=\ell_{i'}$. Further, for each $j\in Y$ we have for some $i\leq s'$, $u_j\in N_i$ and either $\omega(\ell_i)=1$ or  $\omega(\ell_i)\geq 2$. In the latter case we know that there is some $i'\neq i$ such that $u_j\in N_{i'}$. Then this case can be further broken down into two cases: either $i'\leq s'$ or $i'> s'$. Let $X$ be the subset of $j\in Y$ that fall into the case where $i'\leq s'$. Then $|Y| \leq \sum_{i\in Z_1} |N_i| +\sum_{i\in Z_2}|N_i| - |X|$.
 We can also see that $|X|\geq \left(\sum_{i\in Z_2}|N_i|-\sum_{i\in Z'_1}|N_i|\right)/2$. Then recalling that for all $i\in[s]$, $|N_i|\leq 2d$, by \cref{eq:simple_sum} we get that $$|Y|\leq 
\sum_{i\in Z_1} |N_i|+ \left(\sum_{i\in Z_2}|N_i|+\sum_{i\in Z'_1}|N_i|\right)/2\leq d(2|Z_1|+|Z_2|+|Z_1'|) \leq 2dp+d|Z_1|.$$ Finally, by \cref{eq:psi_sum} we get
\begin{align*}
   |R(\tau)|-(n+p)&\geq |R|-|Y|-p-|Z_1|-|Z_2|-(n+p)=p+3dp-|Y|-|Z_1|-|Z_2|
   \\ &\geq (d+1)p-(d+1)|Z_1|-|Z_2|\geq (d-1)p-(d-1)|Z_1|\geq 0. 
\end{align*}

\end{subproof}

Now let us analyze $c$ on this set $R(\tau)$. For each $u\in R(\tau)$ take some maximal $I_u\subset [k]$ such that there is a vertex set $L_u\subset R(\tau)$ of size $k_\tau$ such that for each $i\in I_u$, $c(wu)+\cC_{i}(w)$ is constant across all $w\in R(\tau)\setminus L_u$. By convention, $L_u$ always contains $u$ itself. Further, for any $u\in R(\tau)$ define $R_u\coloneqq R(\tau)\setminus L_u$ and note that $|R_u|= |R(\tau)|-k_\tau$. 

\begin{remark}\label{rem:C-C}
    Fixing $i,j\in I_u$ we observe that for all $w,w'\in R_u$, $$c(uw)+\cC_i(w)-c(uw')-\cC_i(w')=0=c(uw)+\cC_j(w)-c(uw')-\cC_j(w')$$ and so $\cC_i(w)-\cC_i(w')=\cC_j(w)-\cC_j(w')$.
\end{remark}

\begin{claim}
    $|I_u|> k/2$.
\end{claim}

\begin{subproof}[Proof of claim]
For each $i\in [k]$ let $\gamma(i)$ be the most common value of $c(wu)+\cC_{i}(w)$ across all $w\in R(\tau)\setminus\{u\}$. Also, for each $g\in \mathbb{Z}_p$ let $$\alpha(i,g)\coloneqq\{w\in R(\tau)\setminus\{u\}: c(wu)+\cC_{i}(w)\neq g \}$$ be the set of vertices that do not attain $g$. Starting with $i=1$ and $Q_1=I_1=\emptyset$ and working up to $i=k$, at each step let us consider $\alpha(i,\gamma(i))\setminus Q_i$. If it is non-empty choose one of its elements $w'_{i}$, and set $Q_{i+1}=Q_i \cup \{w'_i\}$ and $I_{i+1}=I_i$. If it is empty then instead we set $Q_{i+1}=Q_i$ and $I_{i+1}=I_i\cup \{i\}$. 

After we've finished step $i=k$ we have $|Q_{k+1}|+|I_{k+1}| = k$. If $|Q_{k+1}|< k_\tau$ set $L_u=Q_{k+1}\cup \{u\}$ and $I_u=I_{k+1}$ and observe that for each $i\in I_{k+1}$, $\alpha(i,\gamma(i))\setminus L_u$ is empty. That is, $c(wu)+\cC_{i}(w)$ is constant on $R(\tau)\setminus L_u$. Then, since $|I_{k+1}|=k-|Q_{k+1}| > k/2$ we conclude.

On the other hand, if $|Q_{k+1}| \geq  k_\tau$ then for each $i\notin I_{k+1}$ let us define $\beta(i)\coloneqq c(w'_iu)+\cC_{i}(w'_i)$ and choose some $w_i\in \alpha(i,\beta(i))\setminus Q_{k+1}$ ensuring that all of the $w_i$ are distinct. Observe that this can be done since $\beta(i)\neq \gamma(i)$ and therefore $|\alpha(i,\beta(i))|\geq (|R(\tau)|-1)/2 > 2 k$. Then we have that for each of the  $i\in [k]\setminus I_{k+1}$, $c(w_iu)+\cC_{i}(w_i)\neq \beta(i)= c(w'_iu)+\cC_{i}(w'_i)$. However, recall that since we terminated the process early, we have that $|[k]\setminus I_{k+1}|<k_\tau$. On the other hand, $|[k]\setminus I_{k+1}|=|Q_{k+1}|\geq k_\tau$, yielding a contradiction. 
\end{subproof}

Now we will try again to find our desired embedding functions $f$, $h$, and $h'$, but this time using a different process which utilizes what we know about $R(\tau)$. If this process fails as well, then we will be able to embed $G$ monochromatically into $R(\tau)$ using \Cref{lem:mono}. The new process is much simpler: instead of embedding $U$ one element at a time, we will embed a significant portion of it all at once. Then, instead of trying to embed multiple $v_i$'s at a time, we will embed the first $p$ of them one at a time, requiring that $c(h(v_i)f(N_i))\neq c(h'(v_i)f(N_i))$. Although not strictly necessary, we will attempt this process two times in order to use the vertices of $R(\tau)$ as efficiently as possible.
To begin, let $U'=\bigcup_{i\in [p]} N_i$. For ease of notation, given a function $f':U' \to P\subset R(\tau)$, for each $i\in [p]$ define $$\overline{P_i(f')}
\coloneqq R(\tau) \setminus\left(\bigcap_{r\in f'(N_i)}R_r\right).$$

\begin{claim}\label{R_*}
    Let $P\subset R(\tau)$ be a set of size $|P|=|U'|$ such that there is some $\rho \in [k]$ such that $\rho\in I_u$ for all $u\in P$. Let $f':U'\to P$ be a bijection. Further, let $\cP=\{x_1,...,x_p\}\subset R(\tau)\setminus P$ such that $x_j\notin \overline{P_j(f')}$  for each $j\in [p]$. Then either we can find desired embedding functions $f$, $h$, and $h'$ satisfying the conditions in \Cref{lem:embedding} or there is a subset $\cR(P)\subset R(\tau)$ disjoint from $P$ and $\cP$ of size $$|\cR(P)|\geq |R(\tau)|-(2+2d)p-2dk'$$ such that for all $w\in \cR(P)$, $j\in I_w$, and $w',w''\in R_w\cap \cR(P)$, we have $\cC_j(w')=\cC_j(w'')$.
\end{claim}

\begin{subproof}[Proof of claim]
Start by defining $h(v_i)\coloneqq x_i$ for all $i\in [p]$, and the empty function $h': \emptyset \to \emptyset$. Next, for each $i\in[p]$, going in order, we find some $x_i'\in R(\tau)\setminus (P\cup \cP \cup \{x_j':j<i\})$ such that $c(x_if'(N_i))\neq c(x'_if'(N_i))$ and set $h'(v_i)=x'_i$. Once $i=p$ has been visited, we define $\cP' = \{x_j':j\in [p]\}$ and
we take any disjoint $W,Z',U''\subset R(\tau)\setminus (P\cup \cP \cup \cP')$ of sizes $p$, $|Z|$, and $|U|-|P|$ respectively. Then we finish by extending bijectively $f':Z\to Z'$, $f':U\setminus U'\to U''$, and $h,h':\{v_{p+1},...,v_s\}\to W$ (ensuring that $h$ and $h'$ agree on $\{v_{p+1},...,v_s\}$). If this process succeeds, we are done.

However, if we get stuck at step $i_*\leq p$, define $\cP_{*}\coloneqq \{x_j:j<i_*\}$, $\cP_{*}'\coloneqq \{x'_j:j<i_*\}$,  and $N^*\coloneqq f(N_{i_*})\subset P$. We have that for all $w\in R(\tau)\setminus (P\cup \cP_{*}\cup \cP_*')$, $$c(N^*w)=c(N^*x_{i_*}).$$ 
But recalling that $\rho\in I_u$ for all $u\in P$, we observe that for any $w\in \cR(P)\coloneqq R(\tau)\setminus (P\cup \cP_{*} \cup \cP_{*}' \cup \overline{P_{i_*}(f')})$, $$c(uw)+\cC_\rho(w)=c(ux_{i_*})+\cC_\rho(x_{i_*})$$ and so \begin{equation}\label{N^*}
c(N^*w)+|N^*|\cdot\cC_\rho(w)=c(N^*x_{i_*})+|N^*|\cdot\cC_\rho(x_{i_*})
\end{equation} for all $w\in \cR(P)$. It follows that $|N^*|\cdot\cC_\rho(w)=|N^*|\cdot\cC_\rho(x_{i_*})$ and so $\cC_\rho(w)=\cC_\rho(x_{i_*})$ because $|N^*|=|N_{i_*}|\leq 2d <p$. But also, taking any $w,w'\in \cR(P)$, we have $\cC_\rho(w)=\cC_\rho(x_{i_*})=\cC_\rho(w')$. Then, fixing some $u\in N^*$, by Remark \ref{rem:C-C} and the fact that $\rho \in I_u$, we have $\cC_j(w)=\cC_j(w')$ for all $j\in I_u$. Then since there is some $j\in I_w\cap I_u\neq \emptyset$ (here we use that $|I_w|,|I_u|> k/2$), by an identical argument we get that $\cC_q(w')=\cC_q(w'')$ for all $w',w''\in R_{w}\cap \cR(P)$ and $q\in I_w$. Lastly, note that because $N^*\subset P$, $$|\cR(P)|\geq  |R(\tau)|-|P|-2p+2-2d(k_\tau-1)> |R(\tau)|-(2+2d)p-2dk'.$$
\end{subproof}

We will apply \Cref{R_*} twice in order to avoid the vertex loss incurred by $P$ and $\cP$. First, for each $i\in [k]$ let $T_i\coloneqq \{u\in R(\tau):i\in I_u\}$ and let $T$ be the largest such set. Note that $|T|\geq |R(\tau)|/2> |U'|$ since $$k\cdot |T|\geq \sum_{ i\in[k]}|T_i| = \sum_{u\in R(\tau)} |I_u| \geq |R(\tau)|\cdot k/2.$$
Hence, we can fix some set $A\subset T$ of size $|U'|$, define a bijection $f_a:U'\to A$, and set $\cA=\{x_1,...,x_p\}\subset R(\tau)\setminus A$ such that $x_i\notin \overline{A_i(f_a)}$ for each $i\in [p]$. We know such an $\cA$ exists because $|R(\tau)|-2dp-2dk'>p$. Then we apply \Cref{R_*} to get the set $\cR(A)$ of size $|\cR(A)|>n-(1+2d)p+(2+d)k'$.

Next, for each $i\in [k]$ let us define $T'_i\coloneqq \{u\in \cR(A):i\in I_u\}$ and let $T'$ be the largest such set. Then fix some $B\subset T'$ of size $|B|=|A|=|U'|$, and define the bijection $f_b:U'\to B$. We know that this can be done because $$|T'|\geq |\cR(A)|/2 \geq (n- (1+2d)p +(2+d)k')/2\geq (d+1)^2p-(1+2d)p/2> pd^2+pd\geq 2dp \geq |U'|$$ where we use the fact that $n\geq 2(d+1)^2p$. Then take $\cB=\{b_1,...,b_p\}\subset \cR(A)\setminus B$
such that each $b_i\in \cR(A)\setminus (B\cup \overline{B_{i}(f_b)})$. We know such a set exists because $$|\cR(A)|-|B|-|\overline{B_{i}(f_b)}|\geq n- (1+2d)p +(2+d)k'-2dp-2dk'\geq 2d^2p+p+(2-d)k'> p.$$ 

Applying \Cref{R_*} again, we get the set $\cR(B)\subset R(\tau)$. Then for all $w\in \cR(B)$, $j\in I_w$, and $w',w''\in R_w\cap \cR(B)$, we have $\cC_j(w')=\cC_j(w'')$. The analogous statement also holds for $\cR(A)$. Since $$|\cR(A)\cap \cR(B)|\geq |\cR(A)|-|\cB'|-|\overline{B_{i}(f_b)}|\geq n- (2+4d)p>k',$$ we can observe that for all $w\in \cR \coloneqq \cR(A)\cup \cR(B)$, there is some $u\in R_w\cap \cR(A)\cap \cR(B)$ and therefore $\cC_j(w')=\cC_j(u)$ for all $w'\in R_w\cap \cR$ and $j\in I_w$. Then taking any $w''\in R_w\cap \cR$ we have $\cC_j(w')=\cC_j(u)=\cC_j(w'')$. Thus, recalling that by the definition of $I_w$, $c(ww')+\cC_j(w')=c(ww'')+\cC_j(w'')$, we get that $$c(ww')=c(ww'')\eqqcolon C_w$$ is constant. Finally, observe that because $B,\cB\subset \cR(A) $ and $\cR(B)$, $B$, and $\cB$ are pairwise disjoint, $$|\cR| \geq |\cR(B)| + |B|+|\cB| \geq |R(\tau)|-p+2-2dk' >  n+(2+d)k'.$$

We will now conclude the proof by showing that the graph obtained by restricting $E(\cR)$ to the majority color contains some subgraph $H$ which satisfies the conditions in \Cref{lem:mono}.
Let us begin by considering a partition $D\sqcup D'= \cR$ with $|D|\geq |D'|$ such that $C_x\neq C_{x'}$ for all $x\in D$ and $x'\in D'$. Then each such edge $xx'$ will have either $c(xx')\neq C_x$ or $c(xx')\neq C_{x'}$ (or both). But since each $x$ sees at most $k'$ neighbors in a color other than $C_x$, we have a bound on the number of edges between $D$ and $D'$ to be $k'(|D|+|D'|)$, and so, $|D|\cdot|D'|\leq k'(|D|+|D'|)$ and thus, $$|D'|\leq k'(1 + |D'|/|D|)\leq 2k'.$$ 
Let $g$ be the most common edge color in $E(\cR)$, and define the graph $H$ by $V(H)\coloneqq\{x\in \cR:C_x=g\}$ and $E(H)\coloneqq \{u,u'\in V(H): c(uu')=g\}$. Then by the above,  $|V(H)|\geq |\cR|-2k'\geq n +dk'$ and for each $Y=\{y_1,...,y_d\}\subset V(H)$, $$N(Y)\supset \bigcap_{y_i\in Y}  R_{y_i}$$ and so $|N(Y)|\geq n+dk'-d(k'+1)=n-d$. Thus, we can apply \Cref{lem:mono} to $G$ and $H$ to find a monochromatic embedding of $G$ into $H$, thereby concluding the proof.
\end{proof}

\section{Conclusion}

In their paper, the current author along with Katz, Lian, and Malekshahian \cite{bounded_degree} proved a linear bound for $R(G,\Gamma)$ when $G$ has maximum degree $\Delta$ but allowed $\Gamma$ to be a general abelian group of order up to $|\Gamma|\leq e(G)$.
It would be natural to try to use the techniques from \cite{bounded_degree} to extend the results in this paper to all abelian groups $\Gamma$ with $|\Gamma|\leq e(G)$ (thereby confirming Conjecture 1 in their paper). There are two key techniques that would need to be adapted to our setting. The first deals with the reason behind the restriction $2d<p$. That is, if we allow general abelian groups, we would need to deal with the possibility that there could be $v\in V(G)$ and $0\neq g\in \Gamma$ such that $d(v)\cdot g =0$. As a result, the argument in \cref{N^*} that $|N^*|\cdot \cC_i(w)=|N^*|\cdot \cC_i(w') \implies \cC_i(w)=\cC_i(w')$ (recall that $|N^*|=|N(v_i)|$ for some $v_i\in J$) would no longer hold. Further, it would no longer be possible to use Cauchy-Davenport. Instead, we may replace it with Kneser’s Theorem \cite{Kneser_1953} but this introduces significant complications in insuring that $\{c(\overline{u}):\overline{u}\in \{h(v_1),h'(v_1)\}\times...\times \{h(v_s),h'(v_s)\}\}=\Gamma$. 

Second, in order to allow $|\Gamma|=e(G)$ we could replace $h$ and $h'$ in \Cref{lem:embedding} with a set of functions $h_1,...,h_\alpha$ such that $\sum_{i\in [s]}|\{c(h_j(v_i)f(N_i)):j\in [\alpha]\}|\geq |\Gamma|+s-1$ (this corresponds to the blueprints realized at intensity $\alpha$ in \cite{bounded_degree}). However, it is unclear how the lack of such an embedding would allow us to argue for the existence of a monochromatic copy of $G$ in $R$.

\section*{Acknowledgments}
We would like to thank Matthew Jenssen and Adva Mond for helpful discussion and feedback on a previous draft of this paper.

\bibliographystyle{plain}
\bibliography{bibliography}

\end{document}